\documentclass[a4paper,11pt]{amsart}
\usepackage{color}

\newcommand{\aspas}[1]{``{#1}''}
\usepackage{amsmath,amsfonts,amsthm,amstext}
\usepackage{amsmath,amssymb,indentfirst,epsfig,multicol}
\usepackage{amssymb,latexsym}
\usepackage{lipsum}
\usepackage{epsfig}
\usepackage{graphicx}
\usepackage{amscd}
\usepackage[normalem]{ulem}
\usepackage[english]{babel}
\usepackage[latin1]{inputenc}
\usepackage[all]{xy}
\newtheorem{theorem}{Theorem}[section]
\newtheorem{rema}[theorem]{Remark}

\newtheorem{lema}[theorem]{Lemma}
\newtheorem{exam}[theorem]{Example}
\newtheorem{formulation}[theorem]{Formulation}
\newtheorem{defi}[theorem]{Definition}
\newtheorem{prop}[theorem]{Proposition}
\newtheorem{cor}[theorem]{Corollary}
\begin{document}

\title[Topological complexity of Milnor fibration]{Topological complexity of Milnor fibration}

%    Information for first author:
\author{Cesar A. Ipanaque Zapata}
\address{Departamento de Matem\'atica, IME Universidade de S\~ao Paulo\\
Rua do Mat\~ao 1010 CEP: 05508-090 S\~ao Paulo-SP, Brazil}

%\thanks{The first author was supported in part by NSF Grant \#000000.}

%    Information for second author (if needed):
\author{Taciana O. Souza}
\address{Institute of Mathematics and Statistics, Federal University of Uberl\^andia, Uberl\^andia, Minas Gerais, Brazil.}
\email{tacioli@ufu.br}

\date{}

%    General info
%%%%%%%%%%%%%%%%%%%%%%%%%%%%%%%%%%%%%%%%%%%%%%%%%%%
%   32S55  	Milnor fibration   
%68T40: ROBOTICS
%70Q05  	Control of mechanical systems
%52C35  	Arrangements of points, flats, hyperplanes
%55M30  	Lyusternik-Shnirel'man category of a space, topological complexity à la Farber, topological robotics (topological aspects)
%         Please use the current 2010 Mathematics Subject Classification:             %
%         http://www.ams.org/mathscinet/msc/                                                        %
%         http://www.zentralblatt-math.org/msc/en/                                                 %
%%%%%%%%%%%%%%%%%%%%%%%%%%%%%%%%%%%%%%%%%%%%%%%%%%%

\thanks {The first author would like to thank grant \#2023/16525-7, grant \#2022/16695-7 and grant\#2022/03270-8, S\~{a}o Paulo Research Foundation (FAPESP) for financial support.}

\maketitle

\begin{abstract}
In this paper we discover a connection between the Milnor fibration theory and current research trends in topological robotics. The configuration space and workspace are often described as subspaces of some Euclidean spaces. The work map is a continuous map which assigns to each state of the configuration space the position of the end-effector at that state. The tasking planning problem consists in study the complexity and design of algorithms controlling the task performed by the robot manipulator. We study the tasking planning problem for the Milnor fibration of analytic 
map germs. We see that the tasking planning algorithms strongly depend of the Milnor fibration theorem. We conjecture that the topological complexity of Milnor fibration coincides with the topological complexity of its base. 
\end{abstract}

\section{Introduction}
In this paper \aspas{space} means a topological space, by a ``map'' we will always mean a continuous map, and fibrations are taken in the Hurewicz sense. Given a map $f:X\to Y$, a ``global cross-section'' of $f$ is a map $s:Y\to X$ such that $f\circ s$ is equal to the identity map $1_X:X\to X$.

\medskip Given an autonomous system $\mathcal{S}$, from \cite{bajdRobotics}, recall the following basic notions:
\begin{itemize}
    \item The \textit{configuration space} $C$ is defined as the space of all possible states of $\mathcal{S}$.
    \item The \textit{workspace} $W$ consists of all points that can be reached by the robot end-effector.
    \item The \textit{work map} or \textit{kinematic map} is a map \[f:C\to W\] which assigns to each state of the configuration space the position of the end-effector at that state.
\end{itemize}
The configuration space and workspace are often described as a subspace of some Euclidean space $\mathbb{R}^n$ and $\mathbb{R}^p$, respectively.  For instance, we consider the robot arm as shown in Figure~\ref{fig:RR}. In this case the configuration space is the torus $S^1\times S^1$ and the workspace is the $2$-sphere $S^2$, and the work map is given by $f(\alpha,\beta)=\left(\cos{\alpha}\cos{\beta},\cos{\alpha}\sin{\beta},\sin{\alpha}\right)$, see \cite{pavesic}. 

\begin{figure}[htb] %htb
 \label{fig:RR}
 \centering
 \includegraphics[scale=0.3]{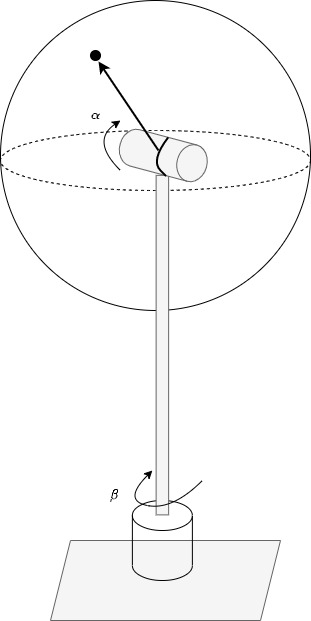}
 \caption{Robot arm (RR).}
\end{figure}

\medskip Let $f:C\to W$ be the kinematic map of a given autonomous system. A \textit{tasking planning algorithm} of $f$, see Pave{\v{s}}i{\'c} \cite{pavesic}, is a function which, to any pair $(c,w)\in C\times W$, assigns a continuous motion $\alpha$ of the system, so that $\alpha$ starts at the given initial configuration $c$ and ends at the desired task $w$ performed by the robot manipulator. The fundamental problem in robotics, \textit{the tasking planning problem}, deals with how to provide, to any kinematic map of a given autonomous system, with a tasking planning algorithm.

\medskip For practical purposes, a tasking planning algorithm should depend continuously on the pair of points $(c,w)$. Unfortunately, if a (global) continuous tasking planning algorithm of $f:C\to W$ exists implies that $f$ admits a global cross-section and $W$ is contractible (see Lemma~\ref{lem:tasking-algorithm-implies-section-contractible}). Yet, if $f$ does not admit a global cross-section or $W$ is not contractible, we could care about finding tasking planning algorithms $s$ defined only on a certain open set of $C\times W$, to which we refer as the domain of definition of $s$. In these terms, a \emph{tasking planner} of $f$ is a set of local continuous tasking planning algorithms whose domains of definition cover $C\times W$. The \emph{topological complexity} of $f$, TC$(f)$, is then the minimal cardinality among tasking planners of $f$ (see Definition~\ref{def:tc}), while a tasking planner of $f$ is said to be \emph{optimal} if its cardinality is TC$(f)$. In the case that $f$ does not admit a tasking planner with a finite number of domains of definition, we write $\text{TC}(f)=\infty$.  The design of explicit tasking planners that are reasonably close to optimal is one of the challenges of modern robotics (see, for example Latombe \cite{latombe2012robot} and LaValle \cite{lavalle2006planning}). 

\medskip In more detail, the components of the tasking planning problem via topological complexity are as follows:
\begin{formulation} Ingredients in the tasking planning problem via topological complexity:
\begin{enumerate}
    \item The kinematic map $f:C\to W$. The topology of this kinematic map is assumed to be fully understood in advance. 
    \item Query pairs $(c,w)\in C\times W$. The point $c\in C$ is designated as the initial configuration of the query. The point $w\in W$ is designated as the goal task.  
\end{enumerate}
In the above setting, the goal is to either describe a tasking planner, i.e., describe
\begin{enumerate}\addtocounter{enumi}{2}
\item An open covering $U=\{U_1,\ldots,U_k\}$ of $C\times W$;
\item For each $i\in\{1,\ldots,k\}$, a local continuous tasking planning algorithm, i.e., a map $s_i\colon U_i\to C^{[0,1]}$ satisfying $$s_i(c,w)\left(0\right)=c,\quad f\left(s_i(c,w)\left(1\right)\right)=w$$ for any $(c,w){}\in U_i$ (here $C^{[0,1]}$ stands for the free-path space on $C$ joint with the compact-open topology),
\end{enumerate}
or, else, report that such an planner does not exist.
\end{formulation}

As a motivation of this work, one has the following observation. 

\begin{rema}\label{rem:sing-higher-tc}
It is known that singularities of the kinematic map implies a reduction in freedom of movement in the working space
that arises in certain joint configurations (see \cite{donelan2010}, \cite{pavesic})  and thus the complexity of tasking planning can be higher. Thus, it is reasonable to ask whether it is possible to lower the topological complexity of work maps by restricting them in a neighbourhood of singularities. 
\end{rema}   

\medskip In the book \cite{Milnor} \emph{Singular Points of Complex Hypersurfaces}, John W. Milnor studied singularities of complex hypersurfaces and showed how to attach a locally trivial smooth fiber bundle to each singular point in order to extract topological data from it. Milnor proved the existence of such fibration structures associated to holomorphic functions and also to real analytic maps, but in the latter case for isolated critical points. Both fibrations are currently known as the Milnor fibrations. The real Milnor fibrations have  been extended in several directions for isolated and not isolated singularities. Some interesting results toward that can be found, for example, in \cite{TCR},  \cite{TR}, \cite{Massey} and \cite{RR}.

\medskip We consider $n> p\geq 2$ and $f:(\mathbb{R}^n,0)\longrightarrow (\mathbb{R}^p,0)$ be an analytic map germ which satisfies the Milnor's conditions $(a)$ and $(b)$ (see \cite[Definition 4.1]{Massey}). Instead to study the topological complexity of $f$ (which can be higher, see Remark~\ref{rem:sing-higher-tc}) we can consider the Milnor fibration as work map: \begin{equation*}
   f_{\mid}:f^{-1}(S^{p-1}_{\eta})\cap B^{n}_{\epsilon}\to S^{p-1}_{\eta},  
\end{equation*}
where $0<\eta\ll\epsilon\leq \epsilon_0,$ and $\epsilon_0$ is a Milnor's radius for $f$ at origin (see Section~\ref{sec:milnor-fibration} or  \cite[Theorem 4.3]{Massey}).  

\medskip We will show that it is possible to lower the topological complexity of work maps by restricting them in a neighbourhood of singularities (see Remark~\ref{rem:lower-tc}(1)), so there is a greater need for work to understand it. The purpose of this work is to study the topological complexity of the Milnor fibration and design their optimal tasking planning algorithms. We see that these tasking planning algorithms strongly depend of the Milnor fibration theorem, see Remark~\ref{rem:depends}. Furthermore, we conjecture that the topological complexity of Milnor fibration coincidences with the topological complexity of its base (which is a sphere).

\medskip The study of topological complexity of the Milnor fibration is quasi non-existent. We note that a previous version of topological complexity of the Milnor fibration appeared in the manuscript \cite{zapata2019preprint}\footnote{The author presents a review of the existence of cross-sections for the Milnor fibration and as a direct consequence he computes the topological complexity and design optimal tasking planning algorithms.}. In this work we present a new theorem of the existence of cross-section for the Milnor fibration. In addition, we note that designing optimal tasking planning algorithms strongly depends of the Milnor fibration theorem. This last is important because it presents a connection between the Milnor fibration theory and current research trends in topological robotics.  

\medskip  The paper is organized as follows: We start with a brief review about the existence of Milnor fibration and its topology under the Milnor's conditions $(a)$ and $(b)$ (Section~\ref{sec:milnor-fibration}). In addition, we present Remark~\ref{explicit-H}. Theorem~\ref{cross-milnor-homotopy-group} presents a characterisation of the existence of cross-sections for Milnor fibrations in terms of homotopy groups of the fiber. In Section~\ref{sec:robotics} we review the notion of topological complexity. The proof of Lemma~\ref{lem:pullback} is very technical and will be use in the last section. Proposition~\ref{sec-milnor-fibration} presents a characterisation of the connectivity of the Milnor fiber in terms of the sectional number. Proposition~\ref{prop:milnor-tube-con-fiber} presents a necessary condition for the connectivity of the Milnor fiber in terms of the singular cohomology of the Milnor tube. Lemma~\ref{lem:tasking-algorithm-implies-section-contractible} presents a motivation for the definition of topological complexity. The algorithms presented in Example~\ref{spheres} are keys to our purpose. Another fundamental result is Proposition~\ref{princ-prop}. In Section~\ref{sec:tc-milnor-fibration} we study the tasking planning problem of Milnor fibrations. For instance, we show that it is possible to lower the topological complexity of work maps by restricting them in a neighbourhood of singularities (Remark~\ref{rem:lower-tc}). Furthermore, we compute the topological complexity of Milnor fibration and we design optimal tasking planning algorithms (Theorem~\ref{thm:principal-theorem}).  

\section{Milnor fibration revisited}\label{sec:milnor-fibration}

Motivated by \cite{Milnor}, Massey establishes Milnor's conditions $(a)$ and $(b)$ as follows (see \cite[Definition 4.1]{Massey}). Let $f = (f_1, \ldots, f_p): (\mathbb{R}^n, 0) \to (\mathbb{R}^p, 0)$ be a non-constant analytic map germ, $2\leq p \leq n-1$, $V:= f^{-1}(0)$ and $\Sigma_{f}$ be the set of critical points of $f.$ Let $r: \mathbb{R}^n \to \mathbb{R}$ be the square of the distance function to the origin, $r(x)=\|x\|^2$, and let $\Sigma_{(f, r)}$ be the set of critical points of the pair $(f,r)$. Note that $\Sigma_{f} \subseteq \Sigma_{(f,r)}.$ Consider that  $B^n_{\epsilon}$ denotes the $n$-dimensional closed ball centered at the origin of radius $\epsilon$ in $\mathbb{R}^n$.

\begin{itemize}
\item[(1)] The map $f$ \textit{satisfies the Milnor condition $(a)$ at origin} if $\Sigma_{f} \subset V$ in a neighborhood of the origin. 

\item[(2)] The map $f$ \emph{satisfies the Milnor condition $(b)$ at origin} if $0$ is isolated in $V\cap \overline{\Sigma_{(f, r)}\setminus V}$ in a neighborhood of the origin, where the \aspas{bar} notation means the topological closure of the space $\Sigma_{(f, r)}\setminus V$ in $\Sigma_{(f,r)}.$

\item[(3)] If $f$ satifies the Milnor's condition $(a)$ \rm(respectively $(b)$\rm), then we say that $\epsilon_{0} >0$ is a \emph{Milnor $(a)$ radius for $f$ at origin} (respectively, Milnor $(b)$ radius for $f$ at origin) provided that for all $0<\epsilon \leq \epsilon_{0}$ one gets
$B_{\epsilon}^n \cap (\overline{\Sigma_{f} \setminus V}) = \emptyset$ \rm(respectively, $B_{\epsilon}^n \cap V \cap  (\overline{\Sigma_{(f, r)} \setminus V}) \subseteq \{0\}$\rm).
\end{itemize}

\medskip Under condition (3) above, we simply say that $\epsilon_{0}$ is a \emph{Milnor radius for $f$ at origin}, if $\epsilon_{0}$ is both a Milnor $(a)$ and Milnor $(b)$ radius for $f$ at origin.

\medskip The Milnor conditions $(a)$ and $(b)$ are enough to ensure the existence of the Milnor fibrations as follows (see \cite[Theorem 4.3]{Massey}). Set $S_{\eta}^{p-1}$ denote the $(p-1)$-dimensional sphere centered at the origin of radius $\eta$ in $\mathbb{R}^p$.

\medskip \textbf{Milnor fibration theorem:} Let $f = (f_1, \ldots, f_p): (\mathbb{R}^n, 0) \to (\mathbb{R}^p, 0)$ be an analytic map germ, and $\epsilon_0$ be a Milnor radius for $f$ at the origin. Then, for each $0< \epsilon \leq \epsilon_0$, there exists $\eta$, $0< \eta \ll \epsilon$, such that

\begin{equation}\label{MasseyII}
f_{|}: B_{\epsilon}^n \cap f^{-1}(S_{\eta}^{p-1}) \to S_{\eta}^{p-1},
\end{equation}
is the projection of a smooth locally trivial fiber bundle. Hence, (\ref{MasseyII}) is a fibration which we will say the real {\it Milnor fibration} in the tube and its fiber $F$ as the real {\it Milnor fiber}. 

\begin{rema}\label{explicit-H}
 Let $f = (f_1, \ldots, f_p): (\mathbb{R}^n, 0) \to (\mathbb{R}^p, 0)$ be an analytic map germ, and $\epsilon_0$ be a Milnor radius for $f$ at the origin. Given any commutative diagram \begin{eqnarray*}
\xymatrix{ &Z \ar[rr]^{\varphi} \ar[d]_{j_0} & & B_{\epsilon}^n \cap f^{-1}(S_{\eta}^{p-1})\ar[d]^{f_{|}} & \\
       &Z\times [0,1]  \ar[rr]_{H} & &  S_{\eta}^{p-1} &  }
\end{eqnarray*}  where $j_0(z)=(z,0)$. By the Milnor fibration theorem, the map $f_{|}: B_{\epsilon}^n \cap f^{-1}(S_{\eta}^{p-1}) \to S_{\eta}^{p-1}$ is a fibration and thus there is a map $\widetilde{H}:Z\times [0,1]\to B_{\epsilon}^n \cap f^{-1}(S_{\eta}^{p-1})$,

\begin{eqnarray*}
\xymatrix{ &Z \ar[rr]^{\varphi} \ar[d]_{j_0} & & B_{\epsilon}^n \cap f^{-1}(S_{\eta}^{p-1})\ar[d]^{f_{|}} & \\
       &Z\times [0,1]\ar@{-->}[urr]^{\widetilde{H}}   \ar[rr]_{H} & &  S_{\eta}^{p-1} &  }
\end{eqnarray*} such that $\widetilde{H}(z,0)=\varphi(z),~\forall z\in Z$ and $f_{|}\circ \widetilde{H}=H$. We do not know an explicit construction for such homotopy $\widetilde{H}$.   
\end{rema}

In \cite[Theorem 8.3]{souza2023}, the authors present new results about the topology of Milnor fibration. For our aim, we recall the following statement. 

\begin{prop}\label{TuboConexo}\cite[Theorem 8.3]{souza2023}
\medskip Let $f = (f_1, \ldots, f_p): (\mathbb{R}^n, 0) \to (\mathbb{R}^p, 0)$, $n > p \geq 2$, be an analytic map germ that satisfies Milnor's conditions $(a)$ and $(b)$ at origin. Assume further that all cells of the link $K= f^{-1} (0) \cap S_{\epsilon}^{n-1}$ have dimensions $\leq n-p-1.$

\begin{itemize}
\item[(1)] If $p\geq 2,$ then the total space of the Milnor fibration (\ref{MasseyII}) is path connected.

\item[(2)] If $p \geq 3$ then the Milnor fiber is path connected.

\item[(3)] If $p=2$ the Milnor fiber is path connected if and only if the Milnor fibration (\ref{MasseyII}) admits a global cross-section.

\end{itemize}
\end{prop}

\medskip Furthermore, one has the following statement.

\begin{prop}\label{group-fiber-cross-section}
    Let $F\hookrightarrow E\stackrel{p}{\to} S^d$ be a fibration ($d\geq 1$) with $F$ and $E$ are path connected. We have that $\pi_{d-1}(F)=0$ if and only if $\pi_{d-1}(E)=0$ and $p$ admits a global cross-section. 
\end{prop}
\begin{proof}
Consider the following piece of the long exact sequence in homotopy of the fibration $p$:
$$\xymatrix{
\pi_d(E) \ar[r]^{p_\#} & \pi_d(S^d) \ar[r] & \pi_{d-1}(F) \ar[r] & \pi_{d-1} (E)\ar[r]^{p_\#} &  \pi_{d-1}(S^d)= 0.
}$$

Suppose $\pi_{d-1}(F)=0$ then $\pi_{d-1} (E)=0$ and $\pi_d(E)\stackrel{p_\#}{\to}\pi_d(S^d)$ is surjective and, therefore, there is a map $\lambda: S^d \to E$ such that $[1_{S^d}] = p_\# ([\lambda]) = [p\circ \lambda]$, where $[1_{S^d}]$ is the homotopy class of the identity map $1_{S^d}$. Thus $\lambda: S^d \to E$ is a global cross-section of $p$.

Conversely, suppose that $\pi_{d-1}(E)=0$ and the Milnor fibration admits a global cross-section $\lambda: S^d \to E$. By the definition $[1_{S^d}] = [p\circ \lambda] = p_\# ([\lambda])$ and, therefore, $p_\#:\pi_d(E)\to \pi_d(S^d)$ is surjective, since $\pi_d(S^d)$ is the infinity cyclic group generated by $[1_{S^d}]$. Thus, since
$$\xymatrix{
\pi_d(E) \ar[r]^{p_\#} & \pi_d(S^d) \ar[r] & \pi_{d-1}(F) \ar[r] & \pi_{d-1}(E)=0.}
$$
is an exact sequence, we conclude that $\pi_{d-1}(F) = 0$.    
\end{proof}

Proposition~\ref{group-fiber-cross-section} together with Proposition~\ref{TuboConexo} imply the following key result. It will be use in our principal Theorem~\ref{thm:principal-theorem}. 

\begin{theorem}\label{cross-milnor-homotopy-group}
Let $f = (f_1, \ldots, f_p): (\mathbb{R}^n, 0) \to (\mathbb{R}^p, 0)$, $n > p \geq 2$, be an analytic map germ that satisfies Milnor's conditions $(a)$ and $(b)$ at origin. Assume further that all cells of $K$ have dimensions $\leq n-p-1.$ Recall that $F\hookrightarrow B_{\epsilon}^n \cap f^{-1}(S_{\eta}^{p-1})\stackrel{f_{|}}{\to} S_{\eta}^{p-1}$ is the Milnor fibration (\ref{MasseyII}). We have that $\pi_{p-2}(F)=0$ if and only if $\pi_{p-2}\left(B_{\epsilon}^n \cap f^{-1}(S_{\eta}^{p-1})\right)=0$ and $f_{|}$ admits a global cross-section.    
\end{theorem}

\section{Topological complexity of a map}\label{sec:robotics}
The concept of topological complexity of a map was introduced by Pave{\v{s}}i{\'c} in \cite{pavesic}, see also \cite{zapatasectional}. Here we recall the basic definitions and properties. 

\subsection{Sectional number} Given a map $f:X \to Y$, the \textit{sectional number} of $f$, denoted by $\text{sec}(f)$, is the smallest number $n$ such that $Y$ can be covered by $n$ open subsets $U_1,\ldots,U_n$ with the property that for every $1\leq i\leq n$ there exists a local cross-section $s_i:U_i\to Y$, i.e., $f\circ s_i$ is equal to the inclusion map $incl_{U_i}:U_i\to Y$. 

\begin{rema}
In the case that $f$ is a fibration, we can relax the equality $f\circ s_i=incl_{U_i}$ by requiring that each of the maps $s_i:U_i\to Y$ satisfies the condition $f\circ s_i\simeq incl_{U_i}$. 
\end{rema}

From \cite[p. 1619]{zapata2022higher}, a \textit{quasi pullback} means a strictly commutative diagram
\begin{eqnarray*}%\label{xfy}
\xymatrix{ \rule{3mm}{0mm}& X^\prime \ar[r]^{\varphi'} \ar[d]_{f^\prime} & X \ar[d]^{f} & \\ &
       Y^\prime  \ar[r]_{\,\,\varphi} &  Y &}
\end{eqnarray*} 
such that, for any strictly commutative diagram as the one on the left hand-side of~(\ref{diagramadoble}), there exists a (not necessarily unique) map $h:Z\to X^\prime$ that renders a strictly commutative diagram as the one on the right hand-side of~(\ref{diagramadoble}). 
\begin{eqnarray}\label{diagramadoble}
\xymatrix{
Z \ar@/_10pt/[dr]_{\alpha} \ar@/^30pt/[rr]^{\beta} & & X \ar[d]^{f}  & & &
Z\rule{-1mm}{0mm} \ar@/_10pt/[dr]_{\alpha} \ar@/^30pt/[rr]^{\beta}\ar[r]^{h} & 
X^\prime \ar[r]^{\varphi'} \ar[d]_{f^\prime} & X \\
& Y^\prime  \ar[r]_{\,\,\varphi} &  Y & & & & Y^\prime &  \rule{3mm}{0mm}}
\end{eqnarray}   

\medskip Furthermore, if the following diagram

\begin{eqnarray*}
\xymatrix{ E^\prime \ar[rr]^{\,\,} \ar[dr]_{p^\prime} & & E \ar[dl]^{p}  \\
        &  B & }
\end{eqnarray*}
commutes, then $\mathrm{sec}\hspace{.1mm}(p^\prime)\geq \mathrm{sec}\hspace{.1mm}(p)$. Also, for any map $p:E\to B$ and any map $f:B^\prime\to B$, note that any local cross-section $s:U\to E$ of $p:E\to B$ induces a local cross-section of the canonical pullback $f^\ast p:B^\prime\times_B E\to B^\prime$, called \textit{the local pullback section} $f^\ast(s):f^{-1}(U)\to B^\prime\times_B E$, simply by defining \[f^\ast(s)(b^\prime)=\left(b^\prime,\left(s\circ f\right)(b^\prime)\right).\]

\begin{eqnarray*}
\xymatrix{ &B^\prime\times_B E \ar[rr]^{ } \ar[d]_{f^\ast p} & & E \ar[d]^{p} & \\
       &B^\prime  \ar[rr]_{f} & &  B & \\
       f^{-1}(U)\ar@{^{(}->}[ru]_{}\ar[rr]_{f}\ar@{-->}@/^10pt/[ruu]^{f^\ast(s)} & &U\ar@{^{(}->}[ru]_{}\ar@{-->}@/^10pt/[ruu]^{s} & & }
\end{eqnarray*} Thus, \begin{align}\label{ineq-canonical}
    \text{sec}(f^\ast p)&\leq \text{sec}(p).
\end{align} 

\medskip The proof of the following statement is very technical and will be use in Proposition~\ref{princ-prop} and Theorem 4.2. 

\begin{lema}\label{lem:pullback}
Let $p:E\to B$ be a map. If the following square
\begin{eqnarray*}
\xymatrix{ E^\prime \ar[r]^{\,\,} \ar[d]_{p^\prime} & E \ar[d]^{p} & \\
       B^\prime  \ar[r]_{\,\, f} &  B &}
\end{eqnarray*}
is a quasi pullback. Then $\text{sec}(p^\prime)\leq \text{sec}(p)$.
\end{lema}
\begin{proof}
Since $p'$ is a quasi pullback, we have the following commutative triangle
\begin{eqnarray*}
\xymatrix{ B^\prime\times_B E \ar[rr]^{\,\,} \ar[dr]_{f^\ast p} & & E^\prime \ar[dl]^{p^\prime}  \\
        &  B^\prime & }
\end{eqnarray*} and thus $\text{sec}(f^\ast p)\geq\text{sec}(p')$. Similarly, since $f^\ast p$ is the canonical pullback, we have the following commutative triangle
\begin{eqnarray*}
\xymatrix{ E^\prime \ar[rr]^{\,\,} \ar[dr]_{p^\prime} & &  B^\prime\times_B E\ar[dl]^{f^\ast p}  \\
        &  B^\prime & }
\end{eqnarray*} and thus $\text{sec}(p')\geq\text{sec}(f^\ast p)$. Hence, the equality $\text{sec}(p')=\text{sec}(f^\ast p)$ holds. By the inequality~(\ref{ineq-canonical}), we obtain $\mathrm{sec}\hspace{.1mm}(p^\prime)\leq \mathrm{sec}\hspace{.1mm}(p)$.  
\end{proof}

The \textit{Lusternik-Schnirelmann category} (LS category) or category of a space $X$ (see \cite{cornea2003lusternik}), denoted by cat$(X)$, is the least integer $n$ such that $X$ can be covered with $n$ open sets which are all contractible within $X$. We have $\text{cat}(X)=1$ iff $X$ is contractible. The LS category is a homotopy invariant, i.e., if $X$ is homotopy equivalent to $Y$ then $\text{cat}(X)=\text{cat}(Y)$. Note that $\text{cat}(S^n)=2$ for any $n\geq 0$.

\medskip For convenience, we record the following standard properties:
\begin{lema}\label{lem:sec-cat}
Let $p:E\to B$ be a map and $h^\ast$ be any multiplicative cohomology theory. Then
\begin{enumerate}
    \item \cite[Theorem 18, p. 108]{schwarz1966} If $p$ is a fibration, then $\text{sec}(p)\leq \text{cat}(B)$.
    \item \cite[Proposi\c{c}\~{a}o 4.3.17-(3), p. 138]{zapata2022} If there are $x_1,\ldots,x_k\in h^\ast(B)$ with \[p^\ast(x_1)=\cdots=p^\ast(x_k)=0 \text{ and } x_1\cup\cdots\cup x_k\neq 0,\] \noindent then \[\text{sec}(p)\geq k+1.\] Where $p^\ast:h^\ast(B)\to h^\ast(E)$ is the induced homomorphism in cohomology.
\end{enumerate}
\end{lema}

\begin{rema}
 In this paper we will use Lemma~\ref{lem:sec-cat}(2) for $h^\ast$ as being the singular cohomology $H^\ast(-;R)$ with any coefficient ring $R$ (as was presented by James in \cite[p. 342]{james1978}).   
\end{rema} 

Proposition~\ref{TuboConexo}(3) implies a characterisation of the connectivity of the Milnor fiber in terms of the sectional number. 

\begin{prop}\label{sec-milnor-fibration}
    Let $f = (f_1, f_2): (\mathbb{R}^n, 0) \to (\mathbb{R}^2, 0)$, $n >2$, be an analytic map germ that satisfies Milnor's conditions $(a)$ and $(b)$ at origin. Assume further that all cells of $K$ have dimensions $\leq n-3.$ Then $\text{sec}(f_{|})=2$ if and only if the Milnor fiber is not path connected. Where $f_{|}$ is the Milnor fibration (\ref{MasseyII}).
\end{prop}
\begin{proof}
  Suppose that $\text{sec}(f_{|})=2$. Then $f_{|}$ does not admit a global cross-section, and by Proposition~\ref{TuboConexo}(3), we have that the Milnor fiber is not path connected. 

  Now, suppose that the Milnor fiber is not path connected. Then, by Proposition~\ref{TuboConexo}(3), we obtain that $f_{|}$ does not admit a global cross-section and thus $\text{sec}(f_{|})\geq 2$. Furthermore, by Lemma~\ref{lem:sec-cat}(1), we have that $\text{sec}(f_{|})\leq \text{cat}(S^1_\eta)$. Recall that $\text{cat}(S^1_\eta)=2$ and hence $\text{sec}(f_{|})=2$.  
\end{proof}

 Lemma~\ref{lem:sec-cat} implies the following example.

\begin{exam}\label{example:cohomology-trivial}
 Let $f = (f_1, \ldots, f_p): (\mathbb{R}^n, 0) \to (\mathbb{R}^p, 0)$, $n > p \geq 2$, be an analytic map germ that satisfies Milnor's conditions $(a)$ and $(b)$ at origin. Recall that $f_{|}: B_{\epsilon}^n \cap f^{-1}(S_{\eta}^{p-1}) \to S_{\eta}^{p-1}$ denotes the Milnor fibration (\ref{MasseyII}). For any coefficient ring $R$, $H^\ast(S^{p-1};R)\cong\dfrac{R[\alpha]}{\langle \alpha^2\rangle}$, where $\alpha\in H^{p-1}(S^{p-1};R)$ is the fundamental class of the sphere. Suppose that $H^{p-1}\left(B_{\epsilon}^n \cap f^{-1}(S_{\eta}^{p-1});R\right)=0$ and thus $f_{|}^\ast(\alpha)=0$ (i.e., $\alpha\in \text{Ker}(f_{|}^\ast)$). Then, by Lemma~\ref{lem:sec-cat}(2), we obtain that $\text{sec}(f_{|})\geq 2$ and hence $\text{sec}(f_{|})=2$ (here we use Lemma~\ref{lem:sec-cat}(1) together with the fact that $\text{cat}(S^{p-1}_\eta)=2$).    
\end{exam}

A direct application of Proposition~\ref{sec-milnor-fibration} together with Example~\ref{example:cohomology-trivial} is the following statement.

\begin{prop}\label{prop:milnor-tube-con-fiber}
   Let $f = (f_1, f_2): (\mathbb{R}^n, 0) \to (\mathbb{R}^2, 0)$, $n >2$, be an analytic map germ that satisfies Milnor's conditions $(a)$ and $(b)$ at origin. Assume further that all cells of $K$ have dimensions $\leq n-3.$ If $H^1\left(B_{\epsilon}^n \cap f^{-1}(S_{\eta}^{1});R\right)=0$ then $\text{sec}(f_{|})=2$ and the Milnor fiber is not path connected. Where $f_{|}$ is the Milnor fibration (\ref{MasseyII}). 
\end{prop}

\subsection{Topological complexity} For a topological space $X$, recall that $X^{[0,1]}$ denote the space of all paths in $X$ with the compact-open topology. Consider the evaluation fibration \begin{equation*}
    e_{0,1}:X^{[0,1]}\to X\times X,~e_{0,1}(\gamma)=\left(\gamma(0),\gamma(1)\right).
\end{equation*} For any map$f:X\to Y$, consider the composite 
\begin{equation*}
    e_{f}:X^{[0,1]}\to X\times Y,~e_{f}=(1_X\times f)\circ e_{0,1}.
\end{equation*}

A \textit{tasking  planning  algorithm} is a function $s\colon X\times Y\to X^{[0,1]}$ satisfying $e_f\circ s=1_{X\times Y}$. Note that, by definition, if there is a tasking planning algorithm of $f$, then $f$ is surjective. Furthermore, we have the following statement.

\begin{lema}\label{lem:tasking-algorithm-implies-section-contractible}
Let $f:X\to Y$ be any map. If a continuous tasking planning algorithm of $f$ exists, then $f$ admits a global cross-section and the space $Y$ is contractible.    
\end{lema}
\begin{proof}
    Suppose that $\sigma:X\times Y\to X^{[0,1]}$ is a global cross-section of $e_f$, that is, $\sigma$ is a map and for any $(x,y)\in X\times Y$, we have that $\sigma(x,y)(0)=x$ and $f(\sigma(x,y)(1))=y$. Set $x_0\in X$ and consider the map $s:Y\to X,~s(y)=\sigma(x_0,y)(1)$. Note that $s$ is a global cross-section of $f$. In addition, the homotopy $H:Y\times [0,1]\to Y$, given by $H(y,t)=f(\sigma(x_0,y)(t))$, satisfies $H(y,0)=f(x_0)$ and $H(y,1)=y$, for any $y\in Y$. Hence, $Y$ is contractible. 
\end{proof}

Lemma~\ref{lem:tasking-algorithm-implies-section-contractible} forces the following definition (cf. \cite{pavesic}, \cite{zapatasectional}).

\begin{defi}\label{def:tc}\rm{
The \textit{topological complexity} TC$(f)$ of a map$f:X\to Y$ is the sectional number of the map $e_f$. In other words, the topological complexity of $f$ is the smallest positive integer TC$(f)=k$ for which  the product $X\times Y$ is covered by $k$ open subsets $X\times Y=U_1\cup\cdots\cup U_k$ such that for any $i=1,2,\ldots,k$ there exists a local cross-section $s_i:U_i\to X^{[0,1]}$ of $e_f$ over $U_i$ (i.e., $e_f\circ s_i=incl_{U_i}$). If no such $k$ exists we will set TC$(f)=\infty$.}
\end{defi}
 
Note that, any collection $s=\{s_i:U_i\to X^{[0,1]}\}_{i=1}^{k}$, where $\{U_i\}_{i=1}^{k}$ is an open covering of $X\times Y$ and each $s_i:U_i\to X^{[0,1]}$ is a local cross-section of $e_f$ over $U_i$; induces a tasking planning algorithm $s:X\times Y\to X^{[0,1]}$ given by $s(x,y)=s_i(x,y)$ where $i$ is the minimal index in such that $(x,y)\in U_i$. Any such tasking planning algorithm $s:=\{s_i:U_i\to X^{[0,1]}\}_{i=1}^{k}$ is called \textit{optimal} if $k=\text{TC}(f)$.

\medskip For convenience, we record the following standard properties of topological complexity (see \cite{zapatasectional}):

\begin{rema}\label{rem:properties}
Let $f:X\to Y$ be a map.
\begin{enumerate}
\item  $\max\{\text{cat}(Y),\text{sec}(f)\}\leq\text{TC}(f)$.
    \item If $f$ admits a global cross-section, then
 $\text{TC}(Y)\leq\text{TC}(f)\leq \text{TC}(X)$.
\end{enumerate}
\end{rema}

The topological complexity of the identity map $1_X:X\to X$, $\text{TC}(1_X)=\text{TC}(X)$, coincides with the topological complexity (a la Farber) of $X$ (see Farber \cite{farber2003topological}). In this case the tasking planning problem is called the \textit{motion planning problem}. 

\medskip For our aim, we recall the topological complexity of spheres together with their optimal motion planning algorithms. Recall the stereographic projection with respect to the north pole $p_N=(0,\ldots,0,1)$: \[p:S^m-\{p_N\}\to \mathbb{R}^m,~(x_1,\ldots,x_{m+1})=\left(\dfrac{x_1}{1-x_{m+1}},\ldots,\dfrac{x_m}{1-x_{m+1}}\right)\] whose inverse $q:\mathbb{R}^m\to S^m-\{p_N\}$ is given by \[q(y_1,\ldots,y_m)=\left(\dfrac{2y_1}{\parallel y\parallel^2+1},\ldots,\dfrac{2y_m}{\parallel y\parallel^2+1},\dfrac{\parallel y\parallel^2-1}{\parallel y\parallel^2+1} \right).\]  

\begin{exam}[Motion planning over spheres]\label{spheres}
From Farber \cite{farber2003topological}, the TC for spheres is as follows:
\[ \text{TC}(S^m)=
\begin{cases}
2, & \hbox{ if $m$ is odd;}\\
3, & \hbox{ if $m$ is even.}
\end{cases} \]

Furthermore:
\begin{enumerate}
    \item[(1)] An optimal motion planning algorithm to $S^m$ with $m$ odd is given by $s:=\{s_i:U_i\to \left(S^m\right)^{[0,1]}\}_{i=1}^{2}$, where \begin{eqnarray*}
U_1&=&  \{(\theta_1,\theta_2)\in S^m\times S^m\mid ~~\theta_1\neq -\theta_2\},\\
U_2&=& \{(\theta_1,\theta_2)\in S^m\times S^m\mid ~~\theta_1\neq \theta_2\},
\end{eqnarray*} with local algorithms: 
$$ 
s_1(\theta_1,\theta_2)(t) = \dfrac{(1-t)\theta_1+t\theta_2}{\parallel (1-t)\theta_1+t\theta_2 \parallel} \text{ for all } (\theta_1,\theta_2)\in U_1;
$$ and for all  $(\theta_1,\theta_2)\in U_2$, 
$$ 
s_2(\theta_1,\theta_2)(t) =  \begin{cases} 
s_1(\theta_1,-\theta_2)(2t), &\hbox{ $0\leq t\leq\dfrac{1}{2}$;}\\
\alpha(-\theta_2,\theta_2)(2t-1), &\hbox{ $\dfrac{1}{2}\leq t\leq 1$,}
\end{cases} $$ where, we consider the subset $F:=\{(\theta_1,\theta_2)\in S^m\times S^m\mid ~~\theta_1= -\theta_2\} $ and for $(\theta_1,\theta_2)\in F$ the map 
$$
\alpha(\theta_1,\theta_2)(t) =  \begin{cases} 
s_1(\theta_1,v(\theta_1))(2t), &\hbox{ $0\leq t\leq\dfrac{1}{2}$;}\\
s_1(v(\theta_1),\theta_2)(2t-1), &\hbox{ $\dfrac{1}{2}\leq t\leq 1$.}
\end{cases} $$ Here, $v$ denote a fixed continuous unitary tangent vector field on $S^{m}$, say $v(x_1,y_1,\ldots,x_\ell,y_\ell)=(-y_1,x_1,\ldots,-y_\ell,x_\ell)$ with $m+1=2\ell$. 
\item[(2)] An optimal motion planning algorithm to $S^m$ with $m$ even is given by $\kappa:=\{\kappa_i:V_i\to \left(S^m\right)^{[0,1]}\}_{i=1}^{3}$, where
\begin{eqnarray*}
V_1&=& \left(S^m\setminus \{p_N\}\right)\times \left(S^m\setminus \{p_N\}\right),\\
V_2&=&  \{(\theta_1,\theta_2)\in S^m\times S^m\mid ~~\theta_1\neq -\theta_2\},\\
V_3&=& \{(\theta_1,\theta_2)\in S^m\times S^m\mid ~~\theta_1\neq \theta_2 \text{ and } \theta_2\neq -1,1\},
\end{eqnarray*} with local algorithms: 
\begin{eqnarray*}
   \kappa_1(\theta_1,\theta_2)(t) &=& q\left((1-t)p(\theta_1)+tp(\theta_2)\right) \text{ for } (\theta_1,\theta_2)\in V_1 \text{ and } t\in[0,1];\\ 
   \kappa_2(\theta_1,\theta_2)(t) &=& \dfrac{(1-t)\theta_1+t\theta_2}{\parallel (1-t)\theta_1+t\theta_2 \parallel} \text{ for all } (\theta_1,\theta_2)\in V_2;
\end{eqnarray*}
and for all  $(\theta_1,\theta_2)\in V_3$, 
$$ 
\kappa_3(\theta_1,\theta_2)(t) =  \begin{cases} 
\kappa_2(\theta_1,-\theta_2)(2t), &\hbox{ $0\leq t\leq\dfrac{1}{2}$;}\\
\beta(-\theta_2,\theta_2)(2t-1), &\hbox{ $\dfrac{1}{2}\leq t\leq 1$,}
\end{cases} $$ where we consider the subset $G:=\{(\theta_1,\theta_2)\in S^m\times S^m\mid ~~\theta_1= -\theta_2 \text{ and } \theta_2\neq -1,1\} $ and for $(\theta_1,\theta_2)\in G$ the map 
$$
\beta(\theta_1,\theta_2)(t) =  \begin{cases} 
\kappa_2(\theta_1,\nu(\theta_1))(2t), &\hbox{ $0\leq t\leq\dfrac{1}{2}$;}\\
\kappa_2(\nu(\theta_1),\theta_2)(2t-1), &\hbox{ $\dfrac{1}{2}\leq t\leq 1$.}
\end{cases} $$ Here, $\nu:S^m\to \mathbb{R}^{m+1}$ denotes the continuous tangent vector field on $S^{m}$, given by $$\nu(x_1,x_2,x_3,\ldots,x_{m},x_{m+1})=(0,-x_3,x_2,\ldots,-x_{m+1},x_{m}).$$ Note that $\nu(1,0,\ldots,0)=0$, $\nu(-1,0\ldots,0)=0$ and $\nu(x)\neq 0$ for any $x\in S^m-\{-1,1\}$.   
\end{enumerate} 
\end{exam}

\medskip The proof of the following proposition is also fundamental for our purposes. 

\begin{prop}\label{princ-prop}
If $p:E\to B$ is a fibration, then \[\text{TC}(p)\leq\text{TC}(B).\] Furthermore, if $p$ admits a global cross-section or $\text{cat}(B)=\text{TC}(B)$, then \[\text{TC}(p)=\text{TC}(B).\]
\end{prop}
\begin{proof} We will check $\text{TC}(p)\leq \text{TC}(B)$. By Lemma~\ref{lem:pullback}, it is sufficient to prove that the following diagram:
 \begin{eqnarray*}
\xymatrix{ E^{[0,1]} \ar[r]^{\,\,p_{\#}} \ar[d]_{e_{p}} & B^{[0,1]} \ar[d]^{e_{0,1}} & \\
       E\times B  \ar[r]_{\,\, p\times 1_{B}} &  B\times B &}
\end{eqnarray*}
 is a quasi pullback. Indeed, for any $\beta:X\to B^{[0,1]}$ and any $\alpha:X\to E\times B$ satisfying $e_{p}\circ\beta=(p\times 1_{B})\circ\alpha$, we need to see that there exists $H:X\to E^{[0,1]}$ such that
$e_{p}\circ H=\alpha$ and $p_{\#}\circ H=\beta$.
\begin{eqnarray*}
\xymatrix{ X \ar@/^10pt/[drr]^{\,\,\beta} \ar@/_10pt/[ddr]_{\alpha} \ar@{-->}[dr]_{H} &   &  &\\
& E^{[0,1]} \ar[r]^{\,\,p_{\#}} \ar[d]^{e_{p}} & B^{[0,1]} \ar[d]^{e_{0,1}} & \\
       & E\times B   \ar[r]_{\quad p\times 1_{B}\quad} &  B\times B &}
\end{eqnarray*} Note that we have the following commutative diagram:
\begin{eqnarray*}
\xymatrix{ X \ar[r]^{\,\,p_{1}\circ\alpha} \ar[d]_{i_0} & E \ar[d]^{p} \\
       X\times I  \ar[r]_{\,\,\beta} &  B}
\end{eqnarray*} where $p_1$ is the projection onto the first coordinate. Because $p$ is a fibration, there exists $H:X\times I\to E$ satisfying $H\circ i_0=p_1\circ\alpha$ and $p\circ H=\beta$, thus we obtain the desired quasi pullback.

The second part of the proposition follows from Remark~\ref{rem:properties}.
  
\end{proof}

%%%%%%%%%%%%%%%%%%%%%%%%%%%%%%%%%%%%%%%%%%%%%%%%%%%%%%%%%%%%%%%%%%%%%%%%%%%%%%%%%%%%%%%%%%%%%%%%%%%%%%%%%%%%%%%%%%%%%%%%%%%%%%%%%%%%%%%%%%%%%%%%%%%%%%%%%%%%%%%%%%%%%%%%%%%%%%%%%%%%%%%%%%%%%%%%%%%%%%
%%%%%%%%%%%%%%%%%%%%%%%%%%%%%%%%%%%%%%%%%%%%%%%%%%%%%%%%%%%%%%%%%%%%%%%%%%%%%%%%%%%%%%%%%%%%%%%%%%%%%%%%%%%%%%%%%%%%%%%%%%%%%%%%%%%%%%%%%%%%%%%%%%%%%%%%%%%%%%%%%%%%%%%%%%%%%%%%%%%%%%%%%%%%%%%%%%%%%
\section{Topological complexity of Milnor fibration}\label{sec:tc-milnor-fibration}

Now, we are ready to study the tasking planning problem of Milnor fibrations. Explicitly, we compute the topological complexity of Milnor fibration and we design optimal tasking planning algorithms (Theorem~\ref{thm:principal-theorem}). By the proof of Theorem~\ref{thm:principal-theorem} together with Remark~\ref{explicit-H} this algorithms strongly depends of the Milnor fibration. 

\medskip We consider $n> p\geq 2$ and $f:(\mathbb{R}^n,0)\longrightarrow (\mathbb{R}^p,0)$ be an analytic map germ which satisfies the Milnor's conditions $(a)$ and $(b)$. In particular, we consider the Milnor fibration as work map: \begin{equation*}\label{milnor-fibration}
   f_{\mid}:B^{n}_{\epsilon}\cap f^{-1}(S^{p-1}_{\eta})\to S^{p-1}_{\eta},  
\end{equation*}
where $0<\eta\ll\epsilon\leq \epsilon_0,$ and $\epsilon_0$ is a Milnor's radius for $f$ at origin. We consider that $E=E(\eta,\epsilon)= B^{n}_{\epsilon}\cap f^{-1}(S^{p-1}_{\eta})$ denotes the Milnor tube.

\begin{rema}\label{rem:lower-tc}
\noindent\begin{enumerate}
    \item[(1)]  By Lemma~\ref{lem:sec-cat}(1) together with Remark~\ref{rem:properties}(1),  Proposition~\ref{princ-prop} and Example~\ref{spheres}, we obtain \[2=\text{cat}(S^{p-1})\leq\text{TC}(f_{\mid})\leq\text{TC}(S^{p-1})\leq 3\] and thus $\text{TC}(f_{\mid})\in\{2,3\}$.
    \item[(2)] By the proof of Proposition~\ref{princ-prop}, any optimal algorithm from $S^{p-1}$ induces an algorithm to $f_{\mid}$ not necessarily optimal. However, it is optimal when $f_{\mid}$ admits a global cross-section or $p$ is even.  
\end{enumerate}  
\end{rema}

Now, we compute the value of $\text{TC}(f_{\mid})$ together with its optimal tasking planning algorithm when \begin{enumerate}
    \item $f:(\mathbb{R}^n,0)\to (\mathbb{R}^p,0)$, $n>p\geq 2$, is an analytic map germ with $p$ even (not necessarily with an  isolated singular point at origin).
    \item $f:(\mathbb{R}^n,0)\to (\mathbb{R}^p,0)$, $n>p\geq 2$, is an analytic map germ with $p$ odd, with an isolated singular point at origin and not empty link $K = f^{-1}(0) \cap S_{\epsilon}^{n - 1}$.
    \item $f:(\mathbb{R}^n,0)\to (\mathbb{R}^p,0)$, $n>p\geq 2$, is an analytic map germ with $p$ odd, with all cells of $K$ have dimensions $\leq n-p-1$ and $\pi_{p-2}(F)=0$. 
\end{enumerate}

\begin{theorem}[Principal Theorem]\label{thm:principal-theorem}
 We have $$\text{TC}(f_{\mid})=\begin{cases}
 2, & \hbox{ if $f$ is as (1);}\\
 3, & \hbox{ if $f$ is as (2) or (3).}
 \end{cases}$$ Furthermore, we design optimal tasking planning algorithms with $2$ and $3$ regions of continuity, respectively.  
\end{theorem}
\begin{proof}
The case (1) follows directly from Proposition~\ref{princ-prop} (by Example~\ref{spheres}, recall that $\text{TC}(S^{p-1})=2$ for any $p$ even). The case (2) and (3) is more difficult, however, by the Milnor fibration theory we have that $f_{\mid}$ admits a global cross-section whenever the link $K = f^{-1}(0) \cap S_{\epsilon}^{n - 1}$ is not empty, see \cite[p. 101]{Milnor} and Theorem~\ref{cross-milnor-homotopy-group}, respectively. Then we can apply again Proposition~\ref{princ-prop} and thus Example \ref{spheres} implies the result.

Now, we present optimal tasking planning algorithms with $2$ and $3$ regions of continuity for the Milnor fibration $f_{\mid}:E(\eta,\epsilon)\to S^{p-1}$, respectively. Recall the following quasi pullbacks:
\begin{eqnarray*}
\xymatrix{ B^\prime\times_B E \ar@/^10pt/[drr]^{\,\,q_2} \ar@/_10pt/[ddr]_{q_1} \ar@{-->}[dr]_{H} &   &  &\\
& \left(E(\eta,\epsilon)\right)^{[0,1]} \ar[r]^{\,\,(f_{\mid})_{\#}} \ar[d]_{e_{f_{\mid}}} & \left(S^{p-1}\right)^{[0,1]} \ar[d]^{e_{0,1}} & \\
       & E(\eta,\epsilon)\times S^{p-1}   \ar[r]_{\quad f_{\mid}\times 1\quad} &  S^{p-1}\times S^{p-1} &}
\end{eqnarray*} where $B^\prime=E(\eta,\epsilon)\times S^{p-1}$, $B=S^{p-1}\times S^{p-1}$ and $E= \left(S^{p-1}\right)^{[0,1]}$. Recall that $H$ is given by the following commutative diagram (see Remark~\ref{explicit-H}):
\begin{eqnarray*}
\xymatrix{ X \ar[r]^{\,\,p_{1}\circ q_1} \ar[d]_{i_0} & E(\eta,\epsilon) \ar[d]^{f_{\mid}} \\
       X\times I\ar@{-->}[ur]_{H}  \ar[r]_{\,\,q_2} &  S^{p-1}}
\end{eqnarray*} where $X=B^\prime\times_B E$.

\textbf{For $p$ even:} From Proposition~\ref{princ-prop}, the optimal algorithm $s:=\{s_i:U_i\to PS^{p-1}\}_{i=1}^{2}$ on the sphere $S^{p-1}$ (see Example \ref{spheres}) induces an optimal tasking algorithm for $f_{\mid}$ with 2 local algorithms, say $\hat{s}:=\{\hat{s}_i:\widehat{V}_i\to \left(E(\eta,\epsilon)\right)^{[0,1]}\}_{i=1}^{2}$, where $\widehat{V}_i=(f_{\mid}\times 1)^{-1}(U_i)\subset E(\eta,\epsilon)\times S^{p-1}$ and $\hat{s}_i(v)=H(v,s_i\circ(f_{\mid}\times 1)(v))$.

\textbf{For $p$ odd and under conditions (2) and (3):} Again, from Proposition~\ref{princ-prop}, the optimal algorithm $\kappa:=\{\kappa_i:V_i\to PS^{p-1}\}_{i=1}^{3}$ on the sphere $S^{p-1}$ (see Example \ref{spheres}) induces an optimal tasking algorithm for $f_{\mid}$ with 3 local algorithms, say $\hat{\kappa}:=\{\hat{\kappa}_i:\widehat{V}_i\to \left(E(\eta,\epsilon)\right)^{[0,1]}\}_{i=1}^{3}$, where $\widehat{V}_i=(f_{\mid}\times 1)^{-1}(V_i)\subset E(\eta,\epsilon)\times S^{p-1}$ and $\hat{s}_i(v)=H(v,\kappa_i\circ(f_{\mid}\times 1)(v))$.
\end{proof}

\begin{rema}\label{rem:depends}
    By the proof of Theorem~\ref{thm:principal-theorem} together with Remark~\ref{explicit-H} these tasking planning algorithms strongly depend of the Milnor fibration. Hence, we stay the following question.
\end{rema} 

\medskip\textbf{Question:} Let $f = (f_1, \ldots, f_p): (\mathbb{R}^n, 0) \to (\mathbb{R}^p, 0)$ be an analytic map germ, and $\epsilon_0$ be a Milnor radius for $f$ at the origin. Given any commutative diagram \begin{eqnarray*}
\xymatrix{ &Z \ar[rr]^{\varphi} \ar[d]_{j_0} & & B_{\epsilon}^n \cap f^{-1}(S_{\eta}^{p-1})\ar[d]^{f_{|}} & \\
       &Z\times [0,1]  \ar[rr]_{H} & &  S_{\eta}^{p-1} &  }
\end{eqnarray*}  where $j_0(z)=(z,0)$. How to construct a map$\widetilde{H}:Z\times [0,1]\to B_{\epsilon}^n \cap f^{-1}(S_{\eta}^{p-1})$,

\begin{eqnarray*}
\xymatrix{ &Z \ar[rr]^{\varphi} \ar[d]_{j_0} & & B_{\epsilon}^n \cap f^{-1}(S_{\eta}^{p-1})\ar[d]^{f_{|}} & \\
       &Z\times [0,1]\ar@{-->}[urr]^{\widetilde{H}}   \ar[rr]_{H} & &  S_{\eta}^{p-1} &  }
\end{eqnarray*} such that $\widetilde{H}(z,0)=\varphi(z),~\forall z\in Z$ and $f_{|}\circ \widetilde{H}=H$?

\medskip This question together with Theorem~\ref{thm:principal-theorem} of course suits best for future implementation-oriented objectives.

\medskip Theorem~\ref{thm:principal-theorem} implies the following statement.

\begin{cor}\label{cor:tc-tcbase}
  Under the hypothesis of Theorem~\ref{thm:principal-theorem} we have that the equality \[\text{TC}(f_{\mid})=\text{TC}(S^{p-1})\] holds. 
\end{cor}

We do not know an analytic map germ $f:(\mathbb{R}^n,0)\to (\mathbb{R}^p,0)$, $n>p\geq 2$, with $p$ odd and not empty link $K = f^{-1}(0) \cap S_{\epsilon}^{n - 1}$, such that $\text{TC}(f_{\mid})=2$. Hence, it yields the following conjecture.

\medskip\textbf{Conjecture:} For any analytic map germ $f:(\mathbb{R}^n,0)\to (\mathbb{R}^p,0)$, $n>p\geq 2$, with $p$ odd and not empty link $K = f^{-1}(0) \cap S_{\epsilon}^{n - 1}$, we have that \[\text{TC}(f_{\mid})=3.\] 

\section*{Conflict of interest}
The authors declare that they have no conflict of interest.

\end{document}